\documentclass[11pt]{amsart}
\usepackage{amssymb}
\usepackage{graphicx} 
\usepackage{enumerate}
\usepackage{mathtools}
\usepackage{color,soul}
\usepackage{cite}
\usepackage{tikz-cd}
\usetikzlibrary{matrix}

 \usepackage{url}
\usepackage{wrapfig}

\usepackage{enumitem}
\newtheorem*{theorem*}{Theorem}

\usepackage{amsmath, amsthm}
 \usepackage{relsize}

\usepackage{color}

\definecolor{lgray}{gray}{0.82}
\definecolor{rredd}{rgb}{120,0,0}

\usepackage[margin=1.2in]{geometry}

\newtheorem{thm}{Theorem}[section]
\newtheorem{prop}[thm]{Proposition}
\newtheorem{rem}[thm]{Remark}
\newtheorem{lem}[thm]{Lemma}
\newtheorem{cor}[thm]{Corollary}
\newtheorem{defn}[thm]{Definition}

\newtheorem{ex}[thm]{Example}

\newcommand{\theoremname}{Theorem:}

\newcommand{\Ok}{\mathcal{O}_K}
\newcommand{\Q}{\mathbb{Q}}

\newcommand{\C}{\mathbb{C}}

\theoremstyle{definition}

\newtheorem{example}[thm]{Example}

\theoremstyle{remark}

\newtheorem{remark}[thm]{Remark}

\numberwithin{equation}{section}

\title{Discrete real specializations of sesquilinear representations of the Braid groups}
\author{Nancy C. Scherich}
\begin{document}
\maketitle

\begin{abstract}
This paper gives a process for finding discrete real specializations of sesquilinear representations of the braid groups using Salem numbers. This method is applied to the Jones and BMW representations, and some details on the commensurability of the target groups are given.\end{abstract}

\section{Introduction}

Representations of the braid groups have attracted attention because of their wide variety of applications from discrete geometry to  quantum computing.   This paper takes the point of view that one should ask structural questions about the image of a braid group representation, in particular whether the image is discrete for specializations of the parameter.  Venkataramana in \cite{VEN} also followed this pursuit for discrete specializations of the Burau representation but with a different approach toward arithmeticity.

Since the Jones representations are used in modeling quantum computations, much work has been done to understand specializations at roots of unity, as explored by Funar and Kohno in \cite{FK2014}, Freedman, Larson and Wang in \cite{FLW}, and many others. However, there seems to be a lack of exploration of the $real$  specializations of these representations. This paper takes a more general approach to find discrete real specializations of any sesquilinear group representation, and shows how this can be applied to representations of the braid groups. The main theorem follows.

\begin{thm}\label{thmdiscrep}
Let $\rho_t:G\rightarrow GL_m(\mathbb{Z}[t,t^{-1}])$ be a group representation with a parameter $t$. Suppose there exists a matrix $J_t$ so that: 
\begin{enumerate}
\item for all $M$ in the image of $\rho_t$, $M^*J_tM=J_t$, where by definition  $M^*(t)=M^\intercal (\frac{1}{t})$,
\item $J_t=(J_{\frac{1}{t}})^\intercal $,
\item $J_t$ is positive definite for $t$ in a neighborhood $\eta$ of 1.
\end{enumerate}
Then, there exists infinitely many Salem numbers $s$, so that the specialization representation $\rho_s$ at $t=s$ is discrete.
\end{thm}

 Further applying a classification theorem of hermitian forms from \cite{SCHAR} proves the following commensurability result of the target groups.

\begin{cor}
For $\rho_t:G\rightarrow SL_{2m+1}(\mathbb{Z}[t,t^{-t}])$ as in Theorem \ref{thmdiscrep}, there exists infinitely many Salem numbers $s$, so that for infinitely many integers $n,k$ the specializations $\rho_{s^k}$ at $t=s^k$ and $\rho_{s^n}$ at $t=s^n$ map into commensurable lattices.
\end{cor}

Squier showed in \cite{sq} that the reduced Burau representation is sesquilinear and satisfies the criteria for the Theorem \ref{thmdiscrep}. Example \ref{Burau4} gives explicit Salem numbers so that specializing the reduced Burau representation to these numbers is discrete. Section \ref{sec:hecke} shows that these results can be applied to all of the irreducible Jones representations, by generalizing Squier's result and proving that all the irreducible summands of the Jones representations fix some nondegenerate sesquilinear form.  Section \ref{sec:BMW} proves that the BMW representations are also sesquilinear, and gives some partial results for discrete specializations.  Lastly, Section \ref{sec:Comm} discusses the lattice structure and commensurability of the target groups for the Salem number specializations.

\section*{Acknowledgements}
The results of this paper formed part of the author's dissertation.  The author would like to thank her advisor, Darren Long, for endless insight, guidance, and support on this project;  Vaughan Jones for helpful conversations; and Warren Scherich for moral support.  This paper was funded in part by the NSF grant DMS 1463740.


\section{Discrete representations using Salem numbers}
\subsection{ Motivation from Squier and the Burau Representation.}\label{motivation}

The reduced Burau representation $\rho_{n,t}:B_{n+1}\rightarrow GL_n(\mathbb{Z}[t,t^{-1}])$ is an irreducible representation of the braid group.  (For the remainder of the paper, the word 'reduced' will be omitted and it is understood that the Burau representation is reduced.) These representations depend on $n$, where $n+1$ is the number of strands, and are parameterized by a variable $t$. Squier showed in \cite{sq} that there is a nondegenerate, $n$ dimensional matrix  $J_{n,t}$ satisfying the following equation
\begin{align}M^*J_{n,t}M=&J_{n,t} \end{align}
\noindent for all $M$ in the image of $\rho_{n,t}$. Here, $M^*$ is the transpose of $M$ after replacing $t$ with $\frac{1}{t}$ in the entries of $M$, $M^*(t):=M(\frac{1}{t})^\intercal $.  $J_{n,t}$ is sesquilinear with respect to $*$, $J_{n,t}^*=J_{n,t}$, and letting $t=x^2$, $J_{n,t}$ is given by the following tri-diagornal matrix

\[J_{n,t}=
\begin{tikzpicture}[baseline=(current bounding box.center)]
\matrix (m) [matrix of math nodes,nodes in empty cells,right delimiter={]},left delimiter={[} ]{
x+x^{-1}  & -1 &     \\
-1  &  &   -1 \\
    & -1 & x+x^{-1} \\
} ;
\draw[dashed] (m-1-1)-- (m-3-3);
\draw[dashed] (m-1-2)-- (m-2-3);
\draw[dashed] (m-2-1)-- (m-3-2);
\end{tikzpicture}
.\]
 When $t$ is a unit complex number, equation (2.1) agrees with the usual unitary relation $(\bar{M})^\intercal M=Id$. Representations that satisfy $(2.1)$ are called \textit{sesquilinear}, and are said to map into a \textit{generalized unitary group}. This terminology will be made precise in the next section.
 
These generalized unitary groups are the key to finding discrete specializations. The method described here is to show that carefully chosen specializations of the parameter $t$ make the entire generalized unitary group discrete, thus making the image of the representation discrete.

Section \ref{sec:hecke} describes how the Burau representation is one summand of the Jones representations, and the sesquilinear property generalizes to all of the Jones representations. 


\subsection{Unitary Groups} In general, unitary groups are matrix groups which respect a form, or inner product. These notions heavily rely on the ring of coefficients and an involution of that ring.  Let $R$ be a ring and $\phi$ an order two automorphism of $R$ . For a matrix $M$ defined over $R$, let  $M^*=\left(M^\phi \right) ^\intercal $, where  $M^\phi $ is matrix after applying $\phi $ to the entries of $M$. For the Burau representation in Section \ref{motivation}, $\phi $ is the map sending $t\mapsto  \frac{1}{t}$ and $R=\mathbb{Z}[t,t^{-1}]$.

\begin{defn} For a matrix $J$ such that $J^*=J$, the \textbf{generalized unitary group}  is 
$$U_m(J,\phi ,R):=\{M\in GL_m(R)|M^* J M=J\}.$$
\end{defn}

Here, $J$ is called a \textbf{sesquilinear form} with respect to $\phi $. For example, in this notation, the familiar unitary group $U_m$ can be written as $U_m(Id, -,\mathbb{C})$ where `$-$' is complex conjugation.  

\subsubsection{Creating Discrete Unitary Groups} 
The Burau representation can be written as $\rho_{n,t}:B_{n+1} \rightarrow U_n(J_{n,t},\phi , \mathbb{Z}[t,t^{-1}])$. With the goal of parameter specialization in mind, the relevant choice for the coefficient ring is a number ring. Discreteness of the unitary group is a delicate relationship between the form $J$ and the algebraic structure of the number ring. More precisely, let $L$ be a totally real algebraic field extension of $\mathbb{Q}$ and let $K$ be a degree two field extension of $L$. Let $\phi $ be the order two generator of $Gal(K/L)$ and let $\Ok$, respectively  $\mathcal{O}_L$, denote the rings of integers of $K$ and $L$. 
\begin{center}
\begin{tikzpicture}
\matrix(a)[matrix of math nodes,
row sep=2em, column sep=3em,
text height=1.5ex, text depth=0.25ex]
{ K\\
 L\\
 };
\path(a-1-1)edge node[right]{2} (a-2-1);
\draw [-] (0,1) to[out=90,in=180] (.25,1.25);
\draw [->] (.25,1.25) to[out=0,in=45] (.3,.85);
\draw[](.25,1.25) node[right]{$\phi $};

\begin{scope}[xshift = -2.3cm]
\matrix(a)[matrix of math nodes,
row sep=2em, column sep=3em,
text height=1.5ex, text depth=0.25ex]
{ \Ok\\
 \mathcal{O}_L\\
 };
\path(a-1-1)edge  (a-2-1);
\draw [-] (0,1) to[out=90,in=180] (.25,1.25);
\draw [->] (.25,1.25) to[out=0,in=45] (.3,.85);
\draw[](.25,1.25) node[right]{$\phi $};
\end{scope}

\begin{scope}[xshift = 2.6cm]
\matrix(a)[matrix of math nodes,
row sep=2em, column sep=3em,
text height=1.5ex, text depth=0.25ex]
{ K^\sigma\subseteq \mathbb{C} \\
 L^\sigma\subseteq \mathbb{R}\\
 };
\path(a-1-1) --(a-2-1);
\draw [-] (-.5,1) to[out=90,in=180] (-.25,1.25);
\draw [->] (-.25,1.25) to[out=0,in=45] (-.15,.9);
\draw[](-.2,1.25) node[right]{$\phi ^\sigma$};
  \end{scope}
\end{tikzpicture}
\end{center}

   Let $\sigma$ be a complex place of $K$, which in this setting is a field homomorphism $\sigma:K\rightarrow \C$ different from $\phi $ and the identity map.  We denote $X^\sigma=\sigma(X)$ for any $X$ in $K$. The algebraic structure is passed along by $\sigma$, meaning  $\mathcal{O}_{K^\sigma}=(\Ok)^\sigma$ is the ring of integers for $K^\sigma$ and $\phi ^\sigma=\sigma\phi \sigma^{-1}$ is an involution on $K^\sigma$. 
   
    Let $J$ be a matrix over $\Ok$ that is sesquilinear with respect to $\phi $.   $J^\sigma$ is sesquilinear with respect to $\phi ^\sigma$. So in particular, $$U_m(J^\sigma,\phi ^\sigma,\mathcal{O}_{K^\sigma})=\{M\in GL_m(\mathcal{O}_{K^\sigma})|({M}^{\phi ^\sigma})^\intercal  J^\sigma M=J^\sigma\}.$$


Since $\sigma$ is a homomorphism, we can see that $( U_m(J,\phi ,\Ok))^\sigma= U_m(J^\sigma,\phi ^\sigma,\mathcal{O}_{K^\sigma})$ by applying $\sigma$ to the equation $J=M^*JM$.


The following results outline compatibility requirements between $J$ and $\mathcal{O}_{K}$, which result in  $U_m(J,\phi , \mathcal{O}_{K})$ as a discrete subgroup of $GL_m(\mathbb{R})$, under the standard euclidean topology.

\begin{prop}\label{bounded}

$U_m(J^\sigma,\phi ^\sigma,\mathcal{O}_{K^\sigma})$  is a bounded group when $J^\sigma$ is positive definite, and $\phi ^\sigma$ is complex conjugation.
\end{prop}

\begin{proof}
 Because $J^\sigma$ is positive definite,  by Sylvester's Law of Inertia and the Gram-Schmidt process, there exists a matrix $Q\in GL_m(\mathbb{C})$ so that $J^\sigma=Q^*IdQ$. This implies that $Q U_m(J^\sigma,\phi ^\sigma,\mathcal{O}_{K^\sigma})Q^{-1}\subseteq U_m(Id,\phi ^\sigma, \mathbb{C})$  which is a subgroup of the compact group $U_m$.\end{proof}

\begin{thm}\label{SUdiscrete}
$U_m(J,\phi ,\Ok)$ is discrete if for every complex place $\sigma$ of $K$, $J^\sigma$ is positive definite and $\phi ^\sigma$ is complex conjugacy.
\end{thm}

\begin{proof}

 Assume that $\{M_n\}$ converges to the identity in $U_m(J,\phi ,\Ok)$.  To show $\{M_n\}$ is eventually constant, we will show that for $n$ large, there are only finitely many possibilities for the entries $(M_n)_{ij}$. 
 
 By assumption, for each $\sigma$ the group $U_m(J^\sigma,\phi ^\sigma, O_{K^\sigma})$ is  bounded by Proposition \ref{bounded}. Also, for every $M_n$, $M_n^\sigma \in U_m(J^\sigma,\phi ^\sigma, O_{K^\sigma})$. So there exists a $B$ so that for large $n$, for all $i,j$, and for all $\sigma$, that $|(M_n^\sigma)_{ij}|<B$.

  For every $M\in U_m(J,\phi ,\Ok)$, the equation $M^*JM=J$ can be rearranged to $JMJ^{-1}=((M^\phi )^\intercal )^{-1}$, showing that $M$ and $((M^\phi )^\intercal )^{-1}$ are simultaneously conjugate. Thus $\{M_n^\phi \}$ also converges to the identity. Convergent sequences are bounded, so for large enough $n$,  $|(M_n)_{ij}|<B$ and $|(M_n)_{ij}^\phi |<B$ for every $ij$-entry.

$L$ is a totally real degree two subfield of $K$, and $\phi $ generates $Gal(K/L)$. So $K$ has one nonidentity real embedding  $\phi $, and all other embeddings are complex. Thus we have shown above that for large $n$ there is a uniform bound $B$ for each entry $(M_n)_{ij}$ and each  Galois conjugate of $(M_n)_{ij}$. There are only finitely many algebraic integers $\alpha$, so that $deg(\alpha)\leq deg(K/\mathbb{Q})$, and with the property that $\alpha$ and all of the Galois conjugates of $\alpha$ have absolute values bounded above by $B$. So there are only finitely many possible entries for $(M_n)_{ij}$, which implies the sequence $\{M_n\}$ is eventually constant.
\end{proof}
\begin{cor}\label{discbraid}
 If $\rho:G\rightarrow U_m(J,\phi ,\Ok)$ is a representation of a group $G$ so that  for every non-identity place $\sigma$ of $K$, $J^\sigma$ is positive definite and $\phi ^\sigma$ is complex conjugacy, then $\rho$ is a discrete representation.
\end{cor}


 



At first glance, the requirements for Corollary \ref{discbraid} seem very specific and perhaps it is doubtful that any such a representation could exist. However, as described in Section \ref{motivation}, Squier showed that the Burau representation maps into a generalized unitary group over $\mathbb{Z}[t,t^{-1}]$, so the next task is to find values of $t$ so that the form and coefficient ring satisfy the specific hypothesis of Corollary \ref{discbraid}. Section \ref{burau} will show how careful specializations of $t$ to certain Salem numbers meet all of the conditions for Corollary \ref{discbraid}. More generally, Section \ref{sec:hecke} will show that every irreducible Jones representation fixes a form $J_t$ with a parameter, and specializations to Salem numbers can also be found to satisfy Corollary \ref{discbraid}. 

\subsection{Salem Numbers}\label{sec:Salem}

 Salem numbers are the key ingredient to the application of Corollary \ref{discbraid}, which requires a real algebraic number field with tight control and understanding of each of its complex embeddings. 

\begin{defn}
  A Salem number s is a real algebraic unit greater than 1, with one real Galois conjugate $\frac{1}{s}$, and all complex Galois conjugates have absolute value equal to 1.
  \end{defn}
  
       \begin{center}

\begin{figure}[h!]
  \begin{tikzpicture}[scale=0.25]
\draw[thick,->] (-4.5,0) -- (4.5,0) ;
\draw[thick,->] (0,-4.5) -- (0,4.5);
\draw[thick] (0,0) circle (3cm);
\filldraw[fill=red!40!white, draw=black] (3.5,0) circle (.2cm)  node[below]{$s$};
\filldraw[fill=red!40!white, draw=black] (2.5,0) circle (.2cm) ;
\filldraw[fill=red!40!white, draw=black] (2,2.3) circle (.2cm);
\filldraw[fill=red!40!white, draw=black] (2,-2.3) circle (.2cm);
\filldraw[fill=red!40!white, draw=black] (-2.8, .9) circle (.2cm);
\filldraw[fill=red!40!white, draw=black] (-2.8, -.9) circle (.2cm);
  \end{tikzpicture}
\caption{ A schematic picture of an order 6 Salem number.}
\end{figure}
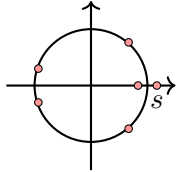
 \end{center}
 
   For example, the largest real root of Lehmer's Polynomial,   called Lehmer's number,$$x^{10}+x^9-x^7-x^6-x^5-x^4-x^3+x+1,$$  is a Salem number.  Trivial Salem numbers of degree two are solutions to $s^2-ns+1$ for $n\in\mathbb{N}$, $n>2$. It is well known that there are infinitely many Salem numbers of arbitrarily large absolute value and degree.  In particular, if $s$ is a Salem number, then $s^{m}$ is also a Salem number for every positive integer $m$. One geometric consequence of this property that powers of Salem numbers are Salem numbers, is that by taking powers, one can control the spatial configuration of the complex Galois conjugates of a Salem number, as described in Lemma \ref{posreal}.

\begin{lem}\label{posreal}
 For any Salem number $s$, and for any interval containing 1 on the complex unit circle, there exists infinitely many integers $m$ so that every complex Galois conjugate of $s^m$ lies in the interval.
\end{lem}

   
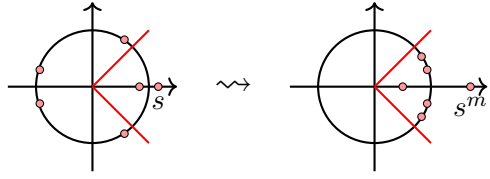
\begin{center}

  \begin{tikzpicture}[scale=0.25]
\draw[thick,->] (-4.5,0) -- (4.5,0) ;
\draw[thick,->] (0,-4.5) -- (0,4.5);
\draw[thick] (0,0) circle (3cm);
\filldraw[fill=red!40!white, draw=black] (3.5,0) circle (.2cm)  node[below]{$s$};
\filldraw[fill=red!40!white, draw=black] (2.5,0) circle (.2cm) ;
\filldraw[fill=red!40!white, draw=black] (1.7,2.5) circle (.2cm);
\filldraw[fill=red!40!white, draw=black] (1.7,-2.5) circle (.2cm);
\filldraw[fill=red!40!white, draw=black] (-2.8, .9) circle (.2cm);
\filldraw[fill=red!40!white, draw=black] (-2.8, -.9) circle (.2cm);

\draw[red,thick] (0,0)--(3,3);
\draw[red,thick] (0,0)--(3,-3);

\node[] at (7.5,0) {\Large$\rightsquigarrow$};

\begin{scope}[xshift=15cm]
\draw[thick,->] (-4.5,0) -- (6,0) ;
\draw[thick,->] (0,-4.5) -- (0,4.5);
\draw[thick] (0,0) circle (3cm);
\filldraw[fill=red!40!white, draw=black] (5.1,0) circle (.2cm)  node[below]{$s^m$};
\filldraw[fill=red!40!white, draw=black] (1.5,0) circle (.2cm) ;
\filldraw[fill=red!40!white, draw=black] (2.5,1.6) circle (.2cm);
\filldraw[fill=red!40!white, draw=black] (2.5,-1.6)circle (.2cm);
\filldraw[fill=red!40!white, draw=black] (2.8, .9) circle (.2cm);
\filldraw[fill=red!40!white, draw=black] (2.8, -.9) circle (.2cm);

\draw[red,thick] (0,0)--(3,3);
\draw[red,thick] (0,0)--(3,-3);
\end{scope}
  \end{tikzpicture}

 \end{center}

\begin{proof}
Let $ e^{i\theta_1}, \cdots ,e^{i\theta_k}$ be all the Galois conjugates of the Salem number $s$ with positive imaginary part.   Suppose that $\prod_{j=1}^k (e^{ i\theta_j })^{m_j}=1$. Let $\phi$ be the automorphism of the Galois closure of $s$ with the property that $\phi(e^{i\theta_1})=s$. Since $\phi$ must permute the Galois conjugates of $s$, for $j\neq 1$, $\phi(e^{i\theta_j})$ is again on the complex unit circle. Thus,

$$1=\phi(\prod_{j=1}^k (e^{i \theta_j })^{m_j})=s^{m_1}\prod_{j=2}^k \phi(e^{i\theta_j })^{m_j}, \text{which implies } \prod_{j=2}^k \phi(e^{i\theta_j })^{m_j}=\frac{1}{s^{m_j}}.$$

Since each $\phi(e^{i\theta_j})$ is a unit complex number,  it must be the case that each $m_j=0$.  This shows that the point $p=(e^{i\theta_1},\cdots,e^{i\theta_k})$ satisfies the criteria for Kronecker's Theorem. In particular, the set $\overline{\{p^m|m\in\mathbb{Z}\}}$ is dense in the torus $T^k$.   \end{proof}

Fixing an arbitrary Salem number $s$,  let $K=\Q(s)$, $L=\Q(s+\frac{1}{s})$, and $\Ok$ be the the ring of integers of $K$. 

 \begin{wrapfigure}{r}{3cm}
\begin{center}
\begin{tikzpicture}
\matrix(a)[matrix of math nodes,
row sep=2em, column sep=3em,
text height=1.5ex, text depth=0.25ex]
{\Q(s)=K\\
 \Q(s+\frac{1}{s})=L\\
\Q\\};
\path(a-1-1)edge node[right]{2} (a-2-1);
\path(a-2-1) edge node[right]{} (a-3-1);
\end{tikzpicture}
\end{center}
\end{wrapfigure}
 Since $s$ and $\frac{1}{s}$ are real and all other Galois conjugates of $s$ are  complex, $K$ has exactly two real embeddings.
For a complex embedding $\sigma$ of $K$, $(s+\frac{1}{s})^\sigma=2$Re$(s^\sigma)$ which is real. This shows that all embeddings of $L$ are real, and that $L$ is a totally real subfield of $K$. Since $s$ is a root of $X^2-(s+\frac{1}{s})X+1$, $K$ is degree two over $L$.

The Galois group of $K/L$ is generated by $\phi $ which maps $s\mapsto \frac{1}{s}$. (This exactly matches the involution $t\mapsto \frac{1}{t}$ needed in the seqsuilinear condition for the Burau representation.)   On the complex unit circle, inversion is the same as complex conjugation. So for the complex embeddings $\sigma$ of $K$, $\phi ^\sigma$ is complex conjugacy.  Notice for a sesquilinear matrix $J_t$ over $\Ok$ with a parameter $t$, specializing $t=s$ leaves $J_s^\sigma$ hermitian.\\

\noindent \textbf{Theorem 1.1:}
\textit{Let $\rho_t:G\rightarrow GL_m(\mathbb{Z}[t,t^{-1}])$ be a representation of a group $G$. Suppose there exists a matrix $J_t$ so that: 
\begin{enumerate}
\item $M^*J_tM=J_t$ for all $M$ in the image of $\rho_t$, 
\item $J_t=(J_{\frac{1}{t}})^\intercal $,
\item $J_t$ is positive definite for $t$ in a neighborhood $\eta$ of 1.
\end{enumerate}
Then, there exists infinitely many Salem numbers $s$, so that the specialization $\rho_s$ at $t=s$ is discrete.}

\begin{proof}
 By Lemma \ref{posreal}, there are infinitely many Salem numbers with the property that all the complex Galois conjugates lie in $\eta$. Let $s$ be one such Salem number. Specializing $t$ to $s$ gives $\rho_s:G\rightarrow U_m(J_s, \phi , O_{\mathbb{Q}(s)})$, where $\phi $ is the usual map given by $s\mapsto \frac{1}{s}$.

Let $\sigma$ be a complex place of $\mathbb{Q}(s)$ which is given by $s\mapsto z$ for $z$ a complex Galois conjugate of $s$. Then $J_s^\sigma=J_z$, and since $z\in\eta$, then $J_z$ is positive definite. By Corollary \ref{discbraid}, the specialization $\rho_s$ at $t=s$ is discrete. \end{proof}

\begin{remark}
 If the representations in Theorem 1.1 all have determinant 1, then the image is more than just discrete, but in fact is a subgroup of a lattice. See Section \ref{sec:Comm} for more details.
\end{remark}

\subsection{ The Burau Representation.}\label{burau}

\begin{prop}
 There are infinitely many Salem numbers $s$ so that the  Burau representation specialized to $t=s$ is discrete.
\end{prop}
\begin{proof}
The specialization of $\rho_{n,1}$ at $t=1$ collapses to an irreducible representation of the symmetric group.  As a representation of a finite group, $\rho_{n,1}$ fixes a positive definite form which is unique up to scaling, by Proposition \ref{onlyform}. At $t=1$, $J_{n,1}$ is positive definite, and the signature of $J_{n,t}$ can only change at zeroes of its determinant.

An inductive computation shows that $\det (J_{n,t})= \frac{x^{2n+2}-1}{x^n(x^2-1)}$ for $t\neq1$, and the zeroes of $\det (J_{n,t})$ occur at  $n+1$'th roots of unity. Thus, $J_{n,t}$ remains positive definite for unit complex values of $t$ with argument less than $\frac{2\pi}{n+1}$. This shows the Burau representation satisfies the criteria of Theorem \ref{thmdiscrep}. 
\end{proof}

\begin{example}\label{Burau4}
The Burau representation $\rho_{4,t}$ of $B_4$ is discrete when specializing $t$ to the following Salem numbers:
 \begin{itemize}
 \item Lehmer's number raised to the powers 16, 32 and 47,
 \item The largest real root of $1 - x^4 - x^5 - x^6 + x^{10}$ raised to the powers 17, 23,  and 43.
 \end{itemize}
\end{example}

\section{ The Hecke Algebras and the Jones Representations}\label{sec:hecke}

The goal of this section is to generalizes Squier's result and show that all of the irreducible Jones representations are sesquilinear, as in the following theorem.

\begin{thm}\label{jonesunitary}
If $\rho$ is an irreducible Jones representation of $B_n$ and $q$ is a generic unit complex number close to 1, then there exists a non-degenerate, positive definite,  sesquilinear matrix $J$ so that for all $M$ in the image of $\rho$, $(M^\phi)^\intercal JM=J$.
\end{thm}

\noindent Then applying Theorem \ref{thmdiscrep} will give the following  discreteness results.



\begin{cor}
For each irreducible Jones representation, there are infinitely many Salem numbers $s$ so that specializing $q=s$,  is a discrete representation.\end{cor}

Before proving the theorem, there is a brief introduction to the Hecke algebras and Young diagrams establishing only pertinent information from this rich subject.

\subsection{Representations of the Hecke Algebras and Young Diagrams}\label{sec:heckeandyoung}

\begin{defn}The \textbf{Hecke algebra (of type $A_n$)}, denoted $H_n(q)$, is the complex algebra generated by invertible elements $g_1,\cdots , g_{n-1}$ with relations
\begin{align}
g_ig_{i+1}g_i&=g_{i+1}g_ig_{i+1}  \text{\hspace{2cm}for all $i<n$}\nonumber \\
g_ig_j&=g_jg_i \text{\hspace{3cm}for } |i-j|>1\nonumber \\
g_i^2&=(1-q)g_i+q  \text{\hspace{1.5cm}for all $i<n$.} (*) \label{eq:hecke}
\end{align}

\end{defn}
Here,  $q$ is a complex parameter. $H_n(q)$ is a quotient of $\mathbb{C}[B_n]$ by relation \ref{eq:hecke}. This quotient can be seen as an eigenvalue condition which forces the eigenvalues of the generators to be $q$ and $-1$. In fact, all of the representations of the braid group with two eigenvalues come from representations of the Hecke algebras, see \cite{JONES}. These representations of the braid group are called the \textbf{Jones representations} which are defined by precomposing a representation of $H_n(q)$ by the quotient map from $\mathbb{C}[B_n]$. Notice that there is a standard inclusion of $H_{n-1}(q)$ into $H_n(q)$ by ignoring the last generator. This gives a standard way to restrict a representation of $H_n(q)$ to a representation of $H_{n-1}(q)$, which respects the restriction of $B_n$ to $B_{n-1}$.

The Hecke algebras come equipped with a natural automorphism, denoted here by $\phi$, which sends $q\mapsto \frac{1}{q}$. 
 Taking $q$ to be a unit complex number, this automorphism becomes complex conjugacy. It is easy to see that when $q=1$, $H_n(q)$ is the complex symmetric group $\mathbb{C}[\Sigma_n]$. What is less obvious but well known is that when $q$ is not a root of unity, $H_n(q)$ is isomorphic to $\mathbb{C}[\Sigma_n]$, see \cite{BOU} pages 54-56. One consequence of this isomorphism is that the parameterization of the irreducible representations of $\Sigma_n$ by  Young diagrams also gives a complete parameterization of the irreducible representations of $H_n(q)$. For a more detailed discussion of Young diagrams see \cite{ZHAO}, and \cite{WENZL} for a construction of the Jones Representations. 

\begin{defn}A \textbf{Young diagram} is a finite collection of boxes arranged in left justified rows, with the row sizes weakly decreasing.
\end{defn}

Every  Young diagram contains sub-Young diagrams by removing boxes in a way that retains the weakly decreasing row length condition.  If $\lambda$ is a  Young diagram with $n$ boxes, then we will call the sub-Young diagrams found by removing one box from $\lambda$ the \textbf{$(n-1)$-subdiagrams of $\lambda$.}

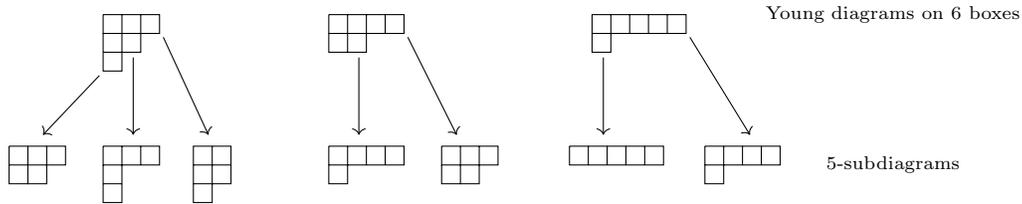
\begin{figure}[h]
\begin{center}
\begin{tikzpicture}
\draw[-](0,3)--(.75,3)--(.75,2.75)--(0,2.75)--(0,3)(.25,3)--(.25,2.25)--(0,2.25)--(0,2.75)(.5,3)--(.5,2.5)--(0,2.5);

\draw[-,xshift=1.2cm,yshift=.75cm](0,0)--(.5,0)--(.5,.5)--(0,.5)--(0,0)(0,0)--(0,-.25)--(.25,-.25)--(.25,.5)(0,.25)--(.5,.25);

\draw[-,yshift=-1.75cm](0,3)--(.75,3)--(.75,2.75)--(0,2.75)--(0,3)(.25,3)--(.25,2.25)--(0,2.25)--(0,2.75)(.5,3)--(.5,2.75)(0,2.5)--(.25,2.5);

\draw[-,yshift=-1.75cm,xshift=-1.25cm](0,3)--(.75,3)--(.75,2.75)--(0,2.75)--(0,3)(.25,3)--(.25,2.5)--(0,2.5)--(0,2.75)(.5,3)--(.5,2.5)--(.25,2.5);

\draw[->](-.05,2.19)--(-.8,1.4);
\draw[->](.4,2.43)--(.4,1.4);
\draw[->] (.8,2.7)--(1.4,1.4);

\draw[-,xshift=3cm](0,3)--(.5,3)--(.5,2.5)--(0,2.5)--(0,3)(.25,3)--(.25,2.5)(0,2.75)--(.5,2.75)(.5,3)--(1,3)--(1,2.75)--(.5,2.75)(.75,3)--(.75,2.75);
\draw[->,xshift=3cm](.4,2.43)--(.4,1.4);
\draw[-,xshift=3cm,yshift=-1.75cm](0,3)--(1,3)--(1,2.75)--(0,2.75)(0,3)--(0,2.5)--(.25,2.5)--(.25,3)(.5,3)--(.5,2.75)(.75,3)--(.75,2.75);
\draw[-,yshift=-1.75cm,xshift=4.5cm](0,3)--(.75,3)--(.75,2.75)--(0,2.75)--(0,3)(.25,3)--(.25,2.5)--(0,2.5)--(0,2.75)(.5,3)--(.5,2.5)--(.25,2.5);
\draw[->,xshift=3.25cm] (.8,2.7)--(1.45,1.4);

\draw[-,xshift=6.5cm](0,3)--(1.25,3)--(1.25,2.75)--(0,2.75)(0,3)--(0,2.5)--(.25,2.5)--(.25,3)(.5,3)--(.5,2.75)(.75,3)--(.75,2.75)(1,3)--(1,2.75);

\draw[-,xshift=8cm,yshift=-1.75cm](0,3)--(1,3)--(1,2.75)--(0,2.75)(0,3)--(0,2.5)--(.25,2.5)--(.25,3)(.5,3)--(.5,2.75)(.75,3)--(.75,2.75);
\draw[->,xshift=6.25cm](.4,2.43)--(.4,1.4);
\draw[-,xshift=6.2cm,yshift=-1.75cm](0,3)--(1.25,3)--(1.25,2.75)--(0,2.75)--(0,3)(.25,3)--(.25,2.75)(.5,3)--(.5,2.75)(.75,3)--(.75,2.75)(1,3)--(1,2.75);
\draw[->,xshift=7cm] (.8,2.7)--(1.6,1.4);
\node[]at (10.5,1){\tiny{5-subdiagrams}};
\node[]at (10.5,3){\tiny{Young diagrams on 6 boxes}};
\end{tikzpicture}
\caption{Example 5-subdiagrams of three different  Young diagrams with 6 boxes.} \label{subdiagrams}
\end{center}
\end{figure}

 A Young diagram is completely determined by its list of $(n-1)$-subdiagrams. In fact, a Young diagram is completely determined by any two of its $(n-1)$-subdiagrams.  To see this, stack any two $(n-1)$-subdiagrams atop each other top left aligned. Each $(n-1)$-subdiagram will contain the missing box from the other $(n-1)$-subdiagram, recovering the original Young diagram. Notice that each pair of the Young diagrams in Figure \ref{subdiagrams} have one 5-subdiagram in common and it is also possible for two different Young diagrams to have the same number of $(n-1)$-subdiagrams.  These $(n-1)$-subdiagrams  also determine representations of the Hecke algebras in a powerful way. The following theorem, due to Jones in \cite{JONES}, states concretely the relationship between Young diagrams and the representations of the Hecke algebras.

\begin{thm}\label{branching}
Up to equivalence, the finite dimensional irreducible representations of $H_n(q)$, for generic $q$, are in one to one correspondence with the  Young diagrams of $n$ boxes. Moreover, if $\rho$ is a representation corresponding to  Young diagram $\lambda$, then $\rho$ restricted to $H_{n-1}(q)$ is equivalent to the representation $\bigoplus_{i=1}^k\rho_{\lambda_i}$ where $\lambda_1,\cdots, \lambda_k$ are all of the $(n-1)$-subdiagrams of $\lambda$ and each $\rho_{\lambda_i}$ is an irreducible representation of $H_{n-1}(q)$ corresponding to $\lambda_i$.
\end{thm}

Here equivalence means the existence of an intertwining isomorphism made precise by the following definition.
\begin{defn}\label{def:equivalentreps}
$\varphi:G\rightarrow GL(V)$ and $\psi:G\rightarrow GL(W)$ are said to be \textbf{equivalent} representations if there exists a linear isomorphism $T:V\rightarrow W$ so that $T\varphi(g)(v)=\psi(g)T(v)$ for all $g\in G$ and $v\in V$, or that the following diagram commutes.
\begin{center}
\begin{tikzcd}
V \arrow[r, "\varphi(g)"] \arrow[d, "T"]
& V \arrow[d, "T"] \\
W \arrow[r, "\psi(g)"]
& W
\end{tikzcd}
\end{center}
\end{defn}
\noindent Choosing bases for $V$ and $W$, the equivalence $T$ gives the matrix equation
$$[T][\varphi(g)][T]^{-1}=[\psi(g)].$$
At the level of matrices, representations are equivalent exactly when they are simultaneously conjugate. In the context of Theorem \ref{branching}, the restriction of $\rho$ to $H_{n-1}(q)$ is equivalent to the representation $\bigoplus_{i=1}^k\rho_{\lambda_i}$, which means there is a change of basis so that the restriction of $\rho$ is block diagonal.

These restriction rules are combinatorially depicted in the lattice of Young diagrams shown in Figure \ref{fig:YoungLattice}. The lines drawn between diagrams in different rows connect the diagrams with $n$ boxes to all of their $(n-1)$-subdiagrams.

\begin{rem}The lattice of Young diagrams has a chain of diagrams with two columns and only one block in the second column, \begin{tikzpicture}[baseline=.2cm]
\draw[-](0,0)--(.2,0)--(.2,.6)--(0,.6)--(0,0)(0,.4)--(.4,.4)--(.4,.6)--(.2,.6);
\node[]at (.1,.3){\tiny$\vdots$};
\end{tikzpicture}. The representations corresponding to these diagrams are the Burau representations. There is a natural symmetry of the lattice of Young diagrams, so depending on the choice of convention, one could define the Burau representations as the diagrams with exactly two rows, and one box in the second row. The Burau representations are shown in red in Figure    \ref{fig:YoungLattice}.
\end{rem}

\begin{figure}[h]
\begin{center}
\begin{tikzpicture}
\draw[-](0,5)--(.25,5)--(.25,5.25)--(0,5.25)--(0,5);
\draw[-](.15,4.9)--(.15,4.6);
\draw[-] (0,4)--(.25,4)--(.25,4.5)--(0,4.5)--(0,4)(0,4.25)--(.25,4.25);
\draw[-](.32,5.1)--(.78,4.6);
\draw[-,red] (.75,4.25)--(1.25,4.25)--(1.25,4.5)--(.75,4.5)--(.75,4.25)(1,4.24)--(1,4.5);
\draw[-](.15,3.9)--(.15,3.6);
\draw[-](.32,4.1)--(.78,3.6);
\draw[](0,3.5)--(.25,3.5)--(.25,2.75)--(0,2.75)--(0,3.5)(0,3)--(.25,3)(0,3.25)--(.25,3.25);
\draw[-] (1,4.1)--(1,3.6);
\draw[-](1.4,4.2)--(2,3.6);
\draw[red] (.75,3.5)--(1.25,3.5)--(1.25,3.25)--(.75,3.25)--(.75,3.5)(1,3.5)--(1,3)--(.75,3)--(.75,3.25);
\draw[] (1.75, 3.5)--(2.5,3.5)--(2.5,3.25)--(1.75,3.25)--(1.75,3.5)(2,3.5)--(2,3.25)(2.25,3.5)--(2.25,3.25);
\begin{scope}[yshift=.25cm]
\draw[yshift=-2.5cm,-](.15,4.9)--(.15,4.6);
\draw[yshift=-1.5cm,-](.32,4.1)--(.78,3.6);
\draw[](0,2)--(.25,2)--(.25,1)--(0,1)--(0,2)(0,1.25)--(.25,1.25)(0,1.5)--(.25,1.5)(0,1.75)--(.25,1.75);
\draw[yshift=-1.5cm,-] (1,4.1)--(1,3.6);
\draw[xshift=-.1cm, yshift=-1.5cm,-](1.4,4.2)--(2,3.6);
\draw[xshift=-.1cm, yshift=-1.5cm,-](1.45,4.45)--(2.75,3.6);
\draw[red] (.75,2)--(1.25,2)--(1.25,1.75)--(.75,1.75)--(.75,2)(1,2)--(1,1.25)--(.75,1.25)--(.75,1.75)(.75,1.5)--(1,1.5);
\draw[] (1.75, 2)--(2.25,2)--(2.25,1.5)--(1.75,1.5)--(1.75,2)(2,2)--(2,1.5)(1.75,1.75)--(2.25,1.75);
\draw[xshift=-.2cm] (3, 2)--(3.75,2)--(3.75,1.75)--(3,1.75)--(3,2)(3.25,2)--(3.25,1.75)(3.5,2)--(3.5,1.75) (3,2)--(3,1.5)--(3.25,1.5)--(3.25,1.75);
\draw[xshift=1.2cm, yshift=-1.5cm,-](1.25,4.4)--(2,3.6);
\draw[xshift=1.2cm, yshift=-1.5cm,-](1.45,4.45)--(2.75,3.6);
\draw[xshift=1cm] (3, 2)--(3.75,2)--(3.75,1.75)--(3,1.75)--(3,2)(3.25,2)--(3.25,1.75)(3.5,2)--(3.5,1.75)(3.75,2)--(4,2)--(4,1.75)--(3.75,1.75);
\node[red] at(1,.8){$\vdots$};
\node[] at (3.5,1.25){$\ddots$};
\end{scope}
\end{tikzpicture}
\caption{Lattice of Young diagrams encoding the restriction rules for the irreducible representations of the Hecke algebras. The Burau representations are shown in red.} \label{fig:YoungLattice}
\end{center}
\end{figure}

\section{Sesqulinear Representations and Contragredients}



A representation is sesquilinear if there exists an invertible  matrix $J$ so that for every $M$ in the image of the representation, the following equation is satisfied
\begin{align}M^*JM=J.\label{eqsesq}\end{align}
Rearranging this equation, we see that $M=J^{-1}((M^\phi)^\intercal )^{-1}J$ showing that $M$ and $((M^\phi)^\intercal )^{-1}$ are simultaneously conjugate. Changing views slightly, consider the following definition.
\begin{defn}
For $\varphi:G\rightarrow GL(V)$ a complex linear representation,  $\tilde{\varphi}:G\rightarrow GL(V^*)$ is called the \textbf{$\phi$-twisted contragredient representation} of $\varphi$ and is given by $\tilde{\varphi}(g)f(v)=f(\varphi(g^{-1})^\phi v)$, for every $g\in G$,$v\in V$ and $f\in V^*$.
\end{defn}

If a basis for $V$ is chosen, then as matrices, $[\tilde{\varphi}(g)]=([\varphi(g)^\phi]^\intercal )^{-1}$. So another way to view a sesquilinear representation is one that is \textit{equivalent to its $\phi$-twisted contragredient}. The reason for using the $\phi$-twisting in addition to the contragrediant is to preserve the character of the representation. For example, the Jones representations have eigenvalues $-1$ and $q$, and the contragredients representations have eigenvalues $-1$ and $\frac{1}{q}$. The involution $\phi$ is necessary to return the $\frac{1}{q}$ eigenvalue back to a ${q}$. 

This viewpoint combined with the following proposition gives a crucial perspective for the proof of Theorem \ref{jonesunitary}.


\begin{prop}\label{onlyform}
 If an absolutely irreducible matrix representation has an invertible matrix $J$ satisfying equation \ref{eqsesq}, then $J$ is unique up to scaling.
\end{prop}
\begin{proof}
Suppose there were two such matrices $J_1$ and $J_2$. Then equation \ref{eqsesq} gives for all matrices $M$ in the representation,
\begin{align*}J_1MJ_1^{-1}=((M^\phi)^\intercal )^{-1}=&J_2MJ_2^{-1}\\
\Rightarrow\text{\hspace{.5cm}} (J^{-1}_1J_2)^{-1}M(J^{-1}_1J_2)=&M.
\end{align*}

This shows that $J_1^{-1}J_2$ is in the centralizer of the entire irreducible representation. Schur's Lemma gives that $J_1^{-1}J_2=\alpha \cdot$Id for some scalar $\alpha$, and finally $J_2=\alpha J_1$.
\end{proof}



\subsection{Proof of Theorem \ref{jonesunitary}}
\begin{lem}\label{equiv}
Every finite dimensional irreducible representation of the Hecke algebra is equivalent to its  $\phi$-twisted contragediant representation, when $q$ is a generic complex number.
\end{lem}

\begin{proof} We can establish this result for $n=3$. There are three non-equivalent irreducible representations of $H_3(q)$ corresponding to the following Young diagrams.
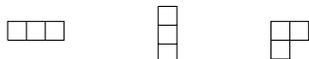
\begin{center}
\begin{tikzpicture}
\draw[-](0,0)--(.75,0)--(.75,.25)--(0,.25)--(0,0);
\draw[-](.25,0)--(.25,.25)(.5,0)--(.5,.25);

\draw[-,xshift=3.5cm](0,0)--(.5,0)--(.5,.25)--(0,.25)--(0,-.25)--(.25,-.25)--(.25,.25);

\draw[-,xshift=2cm,yshift=.2cm](0,.25)--(.25,.25)--(.25,-.5)--(0,-.5)--(0,.25)(0,0)--(.25,0)(0,-.25)--(.25,-.25);
\end{tikzpicture}
\end{center}

Up to equivalence, the first two representations are one dimensional given by $g_i\mapsto q$ and $g_i\mapsto -1$, which are in fact equal to their $\phi$-twisted contragredient representations. The third representation is known to be the Burau representation for $B_3$. As described earlier, Squier showed that the Burau representations are sesquilinear and are therefore equivalent to their $\phi$-twisted contragediant. 


Inductively moving forward, let $\rho:H_n(q)\rightarrow GL(V)$ be a finite dimensional irreducible representation and $\tilde{\rho}$ be the $\phi$-twisted  contragredient representation of $\rho$. Up to equivalence, $\rho$ corresponds to a  Young diagram $\lambda$.  To show that $\rho$ and $\tilde{\rho}$ are equivalent, it suffices to show that both representations correspond to the same $\lambda$. A  Young diagram is completely characterized by its list of $(n-1)$-subdiagrams, which correspond to the restriction of the representation to $H_{n-1}(q)$. So it is enough to show that the restrictions of $\rho$ and $\tilde{\rho}$ correspond to the same list of $(n-1)$-subdiagrams.

Denoting $\rho|=\rho|_{H_{n-1}(q)}$, by Theorem \ref{branching} there is an equivalence $T$ so that 
$$T\rho|(h)T^{-1}=\bigoplus_{i=1}^k\rho_{\lambda_i}(h) \text{\hspace{5mm} for every $h\in H_{n-1}(q)$},$$ where each $\lambda_i$ is an $(n-1)$-subdiagram of $\lambda$, $k$ is the number of $(n-1)$-subdiagrams  of $\lambda$, and $\rho_{\lambda_i}$ is an irreducible representation of $H_{n-1}(q)$ corresponding to $\lambda_i$. Choosing a basis for $V$, the matrix for  $[T\rho|(h)T^{-1}]$ is block diagonal. Taking the $\phi$-twisted contragredient of a block diagonal matrix preserves the block decomposition, which gives 
$$([T^\phi]^\intercal )^{-1}[\tilde{\rho}|(h)][T^\phi]^\intercal =\bigoplus_{i=1}^k[\tilde{\rho}_{\lambda_i}(h)]\text{ for every } h\in H_{n-1}(q).$$

This equation shows that $\tilde{\rho}|$ is equivalent to $\bigoplus \tilde{\rho}_{\lambda_i}$. Since each $\rho_{\lambda_i}$ is an irreducible representation of $H_{n-1}(q)$, we can inductively assume that $\rho_{\lambda_i}$ is equivalent to $\tilde{\rho}_{\lambda_i}$, for all $i\leq k$. Therefore, $\rho_{\lambda_i}$ and $\tilde{\rho}_{\lambda_i}$ correspond to the same Young diagram $\lambda_i$.  Thus the restrictions of $\rho$ and $\tilde{\rho}$ correspond to the same list of $(n-1)$-subdiagrams.

\end{proof}

\noindent \textbf{Theorem} \ref{jonesunitary}. \textit{
If $\rho$ is an irreducible Jones representation of $B_n$ and $q$ is a generic unit complex number close to 1, then there exists a non-degenerate, positive definite,  sesquilinear matrix $J$ so that for all $M$ in the image of $\rho$, $(M^\phi)^\intercal JM=J$.}

\begin{proof}
Let $\rho$ be a finite dimensional irreducible representation of $H_n(q)$ over $V$. By Lemma \ref{equiv}, $\rho$ is equivalent to its $\phi$-twisted contragredient representation $\tilde{\rho}$ by an equivalence $T$. Choose a basis for $V$ and its dual basis for $V^*$, let $ \mathcal{T}$ be the matrix for $T$ with respect to these bases.  We will use this matrix $\mathcal{T}$ to find the desired matrix $J$. Let superscript $*$ denote the $\phi$-twisted transpose of a matrix  to ease computation. For all $g\in H_n(q)$, we get the following matrix equations. \begin{align*}
 \mathcal{T}[\rho(g)] \mathcal{T}^{-1}&=[\tilde{\rho}(g)]=([\rho(g)]^{-1})^*\\
\Rightarrow \text{\hspace{.35cm}}( \mathcal{T}^{-1})^*[\rho(g)]^* \mathcal{T}^* &=[\rho(g)]^{-1}\text{\hspace{2cm} ($\ddagger$)}\\
\Rightarrow  \text{\hspace{.5cm}}\mathcal{T}^*[\rho(g)]( \mathcal{T}^*)^{-1}&=([\rho(g)]^{-1})^*
\end{align*}

This shows that $\mathcal{T}$ and $\mathcal{T}^*$ are two possible forms for $\rho$. By Proposition \ref{onlyform}, $ \mathcal{T}=\alpha  \mathcal{T}^*$ for some $\alpha\in \C$. Applying $*$ again gives $\mathcal{T}=\alpha \alpha^* \mathcal{T}$ and $\alpha\alpha^*=1$. 


Define $J=\beta\mathcal{T}+\beta^*\mathcal{T}^*=(\alpha\beta+\beta^*)\mathcal{T}^*$ where $\beta$ is as follows. (The need for $\beta$ is to ensure that $J$ is invertible.) If $\alpha \neq -1$, let $\beta=1$ which gives that $\det J=\det((\alpha+1)\mathcal{T})$ which is nonzero.  If $\alpha=-1$, let $\beta\in \mathbb{C}$ so that $\beta^*\neq \beta$. Then $\det J=\det [(\alpha\beta+\beta^*)\mathcal{T}^*]=\det [(-\beta+\beta^*)\mathcal{T}]$ is nonzero. So in both cases, $J$ is invertible.


Secondly, $J$ is sesqulinear, that is $J^*=(\beta\mathcal{T}+\beta^*\mathcal{T}^*)^*=\beta^*\mathcal{T}^*+\beta\mathcal{T}=J$. If $M$ is a matrix in the image of $\rho$, rearranging equation ($\ddagger)$ gives $M^*\mathcal{T}^*M=\mathcal{T}.$ So, inserting $J$ gives$$M^*JM=M^*(\alpha\beta+\beta^*)\mathcal{T}^*M=(\alpha\beta+\beta^*)M^*\mathcal{T}^*M=(\alpha\beta+\beta^*)\mathcal{T}=J.$$

It remains to show that $J$ is positive definite. Taking $q=1$, $\rho$ is an irreducible representation of the symmetric group $\Sigma_n$. As a linear representation of a finite group, $V$ admits an inner product that is invariant under the action of $\Sigma_n$, given by a positive definite nondegenerate matrix $\hat{J}$. Proposition \ref{onlyform} guarantees that $\hat{J}$ is unique up to scaling. Since $J|_{q=1}$ is also a form for this representation, it must be that $\hat{J}$ is a multiple of $J|_{q=1}$, which gives that $J$ is positive definite for $q=1$. Since $J$ is Hermitian for unit complex $q$, it has real eigenvalues, and continuity of the determinant map finally gives that either $J$ or $-J$ is positive definite for $q$ close to 1.

\end{proof}

\begin{cor}\label{Jonesdiscrete}
For each irreducible Jones representation, there are infinitely many Salem numbers $s$ so that specializing $q=s^m$, for some $m$,  is a discrete representation.\end{cor}

\begin{ex} \label{33}

Given explicit matrices $S_1,\cdots, S_{n-1}$ for the generators of an irreducible Jones representation of $B_n$, $J$ can be directly computed by solving  the \textit{linear} systems$$S_i^*JS_i-J=0$$ for $1\leq i\leq  n-1$. The form can be made Hermitian by taking $J+J^*$.

On page 362 of \cite{JONES}, Jones gives explicit matrices for the irreducible Jones representation of $B_6$  corresponding to the Young diagram 
\begin{tikzpicture}
\draw[-](0,0)--(.4,0)--(.4,.6)--(0,.6)--(0,0)(.2,0)--(.2,.6)(0,.2)--(.4,.2)(0,.4)--(.4,.4);
\end{tikzpicture}, which has only  one $5$-subdiagram, 
\begin{tikzpicture}
\draw[-](0,0)--(.2,0)--(.2,.6)--(0,.6)--(0,0)(.4,.2)--(.4,.6)(0,.2)--(.4,.2)(0,.4)--(.4,.4)(0,.6)--(.4,.6);
\end{tikzpicture}. The restriction to $B_5$ is also irreducible, and the same form $J$ will work for both the restriction and the full representation. Solving four linear equations as described above yields, 
 \[ J=\left( \begin{array}{ccccc}
\frac{(1 + q)^2}{q}& -1 - q& 2& -1 - q& -1 - q\\
 -\frac{1 + q}{q}& \frac{1 + q + q^2}{q}& -\frac{1 + q}{q}& 1& 1\\
  2& -1 - q& \frac{(1 + q)^2}{q}& -1 - q& -1 - q\\
   -\frac{1 + q}{q}& 1& -\frac{1 + q}{q}& \frac{1 + q + q^2}{q}& 1\\
 -\frac{1 + q}{q} & 1& -\frac{1 + q}{q}& 1& \frac{1 + q + q^2}{q}
  \end{array} \right).\]

\end{ex}



\section{The BMW Representations }\label{sec:BMW}

 The BMW algebras $C_n(l,m)$ are a two parameter family of algebras with $n-1$ generators \cite{BW,MURA,ZIN}. The invertible generators are denoted $G_1,\cdots ,G_{n-1}$ which satisfy the following relations in terms of non-invertible elements denoted by $E_i$ as follows:
\begin{align}
G_iG_j&=G_jG_i\text{\hspace{.5cm}for } |i-j|>1\label{eq:bmwfarrel}\\
G_iG_{i+1}G_i&=G_{i+1}G_iG_{i+1}\label{eq:bmwbraidrel}\\
G_i^2&=m(G_i+l^{-1}E_i)-1.\label{eq:bmw}
\end{align}

There are many additional relations  which can be found in \cite{BW}. Equations \ref{eq:bmwfarrel} and \ref{eq:bmwbraidrel} show that $C_n(l,m)$ can be seen as a quotient of $\C[B_n]$ and there is a standard homomorphism sending $\sigma_i\mapsto G_i$. Also, $C_{n-1}(l,m)\subseteq C_n(l,m)$ and this respects the usual inclusion $B_{n-1}\subseteq B_n$. A representation of the BMW algebra induces a representation of the braid group by mapping $\sigma_i\mapsto \rho(G_i)$. Notice, if $E_i=0$ then equation \ref{eq:bmw} reduces to $G_i^2=mG_i-1$ which is very close to the defining relation for the Hecke algebras. The Hecke algebras are indeed isomorphic to a quotient of the BMW algebra best described by $E_i\mapsto 0$ and $G_i\mapsto lg_i$. This copy of the Hecke algebra inside $C_{n}(l,m)$ is denoted by  $H_n$.

 Analogously to the Jones representations, the irreducible representations of the BMW algebras are parameterized by a Bratteli diagram whose vertices are Young Diagram as shown in Figure \ref{fig:bmwlattices}, but the restriction rule is quite different from the standard Young lattice \cite{BW}. \\

\noindent \textit{ \textbf{BMW restriction rule}: A Young diagram $\lambda_n$ in row $n$ is connected to a Young diagram $\lambda_{n+1}$ in row $n+1$  if $\lambda_{n+1}$ is obtained from $\lambda_n$ by adding or removing a single box.}\\
 
 As depicted in Figure \ref{fig:bmwlattices}, the standard Young lattice occurs in the Bratteli diagram and corresponds to $H_n$, the subalgebra of $C_n(l,m)$ isomorphic to the Hecke algebras. The induced representations of the braid group coming from the subalgebra $H_n$ are the Jones representations \cite{ZIN}.

\begin{figure}[h]
\begin{center}
\begin{tikzpicture}

\draw[xshift=-.75cm, red](0,5)--(.25,5)--(.25,5.25)--(0,5.25)--(0,5);
\draw[] (-.65,4.95)--(-1.25,4.5);
\draw[] (-.65,4.95)--(.2,4.55);
\draw[] (-.65,4.95)--(1,4.55);

\node[red] at (-1.5,4.3) {$\emptyset$};
\draw[-,] (0,4)--(.25,4)--(.25,4.5)--(0,4.5)--(0,4)(0,4.25)--(.25,4.25);
\draw[-] (.75,4.25)--(1.25,4.25)--(1.25,4.5)--(.75,4.5)--(.75,4.25)(1,4.24)--(1,4.5);

\begin{scope}[yshift=-.4cm]
\node[] at (-1.5,1.8) {$\emptyset$};
\end{scope}

\draw[] (-1.25,4.1)--(-.665,3.3);

\draw[] (.2,3.9)--(-.635,3.3);
\draw[] (.2,3.9)--(2,3.3);
\draw[] (.2,3.9)--(3,3.3);

\draw[] (1,4.15)--(-.59,3.3);
\draw[] (1,4.15)--(3.1,3.3);
\draw[] (1,4.15)--(4,3.35);

\draw[xshift=-.75cm,yshift=-2cm,red](0,5)--(.25,5)--(.25,5.25)--(0,5.25)--(0,5);

\begin{scope}[xshift=2cm, yshift=-.25cm]
\draw[](0,3.5)--(.25,3.5)--(.25,2.75)--(0,2.75)--(0,3.5)(0,3)--(.25,3)(0,3.25)--(.25,3.25);
\draw[] (.75,3.5)--(1.25,3.5)--(1.25,3.25)--(.75,3.25)--(.75,3.5)(1,3.5)--(1,3)--(.75,3)--(.75,3.25);
\draw[] (1.75, 3.5)--(2.5,3.5)--(2.5,3.25)--(1.75,3.25)--(1.75,3.5)(2,3.5)--(2,3.25)(2.25,3.5)--(2.25,3.25);
\end{scope}

\draw[] (-.65,2.9)--(-1.35,1.6);
\draw[] (-.65,2.9)--(.1,1.55);
\draw[] (-.65,2.9)--(1,1.55);

\draw[](2.15 ,2.45)--(.15,1.55);
\draw[](2.15 ,2.45)--(2.15,1.55);
\draw[](2.15 ,2.45)--(2.9,1.55);

\draw[](3 ,2.65)--(.25,1.55);
\draw[](3 ,2.65)--(1.1,1.55);
\draw[](3 ,2.65)--(3,1.55);
\draw[](3 ,2.65)--(4,1.55);
\draw[](3 ,2.65)--(5,1.55);

\draw[] (4.12 ,2.95)--(1.2,1.55);
\draw[] (4.12 , 2.95)--(5.1,1.55);
\draw[] (4.12 , 2.95)--(6.1,1.55);


\begin{scope}[yshift=-3cm]
\draw[-] (0,4)--(.25,4)--(.25,4.5)--(0,4.5)--(0,4)(0,4.25)--(.25,4.25);
\draw[-,red] (.75,4.25)--(1.25,4.25)--(1.25,4.5)--(.75,4.5)--(.75,4.25)(1,4.24)--(1,4.5);
\end{scope}

\begin{scope}[yshift=-.5cm, xshift=2cm]
\draw[](0,2)--(.25,2)--(.25,1)--(0,1)--(0,2)(0,1.25)--(.25,1.25)(0,1.5)--(.25,1.5)(0,1.75)--(.25,1.75);
\draw[] (.75,2)--(1.25,2)--(1.25,1.75)--(.75,1.75)--(.75,2)(1,2)--(1,1.25)--(.75,1.25)--(.75,1.75)(.75,1.5)--(1,1.5);
\draw[] (1.75, 2)--(2.25,2)--(2.25,1.5)--(1.75,1.5)--(1.75,2)(2,2)--(2,1.5)(1.75,1.75)--(2.25,1.75);
\draw[xshift=-.2cm] (3, 2)--(3.75,2)--(3.75,1.75)--(3,1.75)--(3,2)(3.25,2)--(3.25,1.75)(3.5,2)--(3.5,1.75) (3,2)--(3,1.5)--(3.25,1.5)--(3.25,1.75); 
\draw[xshift=1cm] (3, 2)--(3.75,2)--(3.75,1.75)--(3,1.75)--(3,2)(3.25,2)--(3.25,1.75)(3.5,2)--(3.5,1.75)(3.75,2)--(4,2)--(4,1.75)--(3.75,1.75); 
\node[] at(1,.9){$\vdots$};
\node[] at(-3,.9){$\vdots$};
\node[] at (4.5,.9){$\ddots$};
\end{scope}

\end{tikzpicture}
\caption{Bratteli diagram for the restriction rules of the irreducible representations of the BMW algebras. The Lawrence-Krammer representation is shown in red.}\label{fig:bmwlattices}
\end{center}
\end{figure}
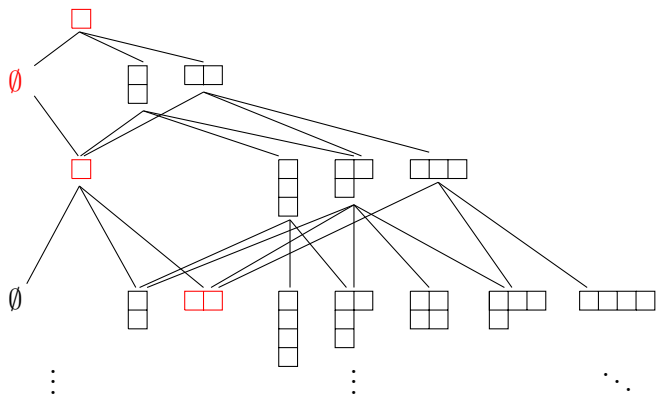

Notice that any $\lambda_{n+1}$ in the $n+1$'st row is completely determined by the the set of diagrams in the $n$'th row connecting to $\lambda_{n+1}$. We can describe this property by saying a diagram is determined up to equivalence by the restriction rule (equivalent as in the sense of Definition \ref{def:equivalentreps} from Section \ref{sec:heckeandyoung}). Similar to how the the Burau representation is one irreducible summand of the Jones representations, Zinno proved in \cite{ZIN} that the Lawrence-Krammer representation is one summand of the BMW representations, colored red in Figure \ref{fig:bmwlattices}.

\begin{ex}\label{ex:onedimBMW}

The three 1-dimensional BMW representations are $\varphi_1(G)=\frac{1}{l}$, $\rho_1(G)= \frac{1}{2}(m-1\sqrt{-4+m^2}) $ and  $\rho_2(G)= \frac{1}{2}(m+1\sqrt{-4+m^2}) $.  Since $\rho_i(E)=0$, both $\rho_i$'s are representations of the Hecke algebra and correspond to the two Young diagrams in the second row of the Bratteli diagram. The third representation $\varphi_1$ corresponds to the  $\emptyset$ diagram.

\end{ex}


\begin{ex}
In \cite{BW}, Birman and Wenzl computed the representation of $B_3$ corresponding to the single box Young diagram.
 \[\sigma_i\mapsto \left( \begin{array}{ccc}
l^{-1} & m & 0  \\
0 & m  & 1\\
0&-1& 0 \end{array} \right), 
\sigma_2\mapsto
\left( \begin{array}{ccc}
0& 0 & -1 \\
0 & l^{-1}  & l^{-1}m \\
1&0& m \end{array} \right)
\]
\end{ex}

\subsection{Sesquilinearity}

 Budney proved that the Lawrence Krammer representation is sesquilinear \cite{BUD}. This section extends Budney's results to all of the BMW representations as stated in the following theorem.

\begin{thm}\label{thm:bmwsesq}
If $\rho$ is an irreducible BMW representation of $B_n$  then there exists a non-degenerate, sesquilinear matrix $J$ so that for all $M$ in the image of $\rho$, $M^* JM=J$.
\end{thm}

To make sense of the $*$ operation, the relevant involution for the BMW algebra is $l\mapsto \frac{1}{l}$, $m\mapsto m$, and $\alpha\mapsto \frac{1}{\alpha}$ where $m=\alpha+\frac{1}{\alpha}$. We will denote this involution by $\phi$. So to show sesquilineararity in this context is to show that the representations are equivalent to their $\phi$-twisted contragrediant representation using $\phi$ to define $*$.

The proof of Theorem \ref{thm:bmwsesq} is exactly analogous to the proof of Theorem \ref{jonesunitary} showing that the Jones representations are sesquilinear, excluding the positive definite argument. It is only necessary to prove the following Lemma.

\begin{lem}
If $\rho$ is an irreducible BMW representation of $B_n$, then $\rho$ is equivalent to its $\phi$-twisted contragredient representation.
\end{lem}

\begin{proof} Analogously to Lemma \ref{equiv}, we will prove this result by induction on $n$. 

Let $\rho$ be an irreducible BMW representation of $B_2$. As shown in Example \ref{ex:onedimBMW}, there are three possible 1-dimensional representations given by $\varphi_1(G)=\frac{1}{l}$, $\rho_1(G)= \frac{1}{2}(m-1\sqrt{-4+m^2}) $ and  $\rho_2(G)= \frac{1}{2}(m+1\sqrt{-4+m^2}) $. 

In the diagonalization process to compute these representations, we introduced a square-root term $ \sqrt{-4+m^2}$. Extending $\phi$ to the field including this term, we define $\phi(\sqrt{-4+m^2})=-\sqrt{-4+m^2}$. Each of these one-dimensional representations are equal to their $\phi$-twisted contragredient representation.

Inductively moving forward, let $\rho$ be an irreducible BMW representation of $B_n$, and $\tilde{\rho}$ be the $\phi$-twisted contragadient of $\rho$. Let $\rho$ correspond to Young diagram $\lambda_1$ and $\tilde{\rho}$ correspond to $\lambda_2$. Since $\lambda_1$ and $\lambda_2$ are completely determined by the BMW restriction rule, the inductive step is the exactly the same as in Lemma \ref{equiv}.
\end{proof}

\subsection{} 


\noindent\textit{\textbf{Positive Definiteness Conjecture:} The forms for the BMW representations, found in Theorem \ref{thm:bmwsesq}, are positive definite for specializations in some open neighborhood in $\mathbb{C}^2$.}\\

This conjecture has been experimentally verified for several of the smaller indexed BMW representations as described in Example \ref{ex:bmw2} to follow, but has not been proven in general. Since the proof of Theorem  \ref{thm:bmwsesq} was completely analogous to that of Theorem \ref{jonesunitary}, at first glance there is hope to repeat the positive definiteness argument that worked for the Jones representations. However, there is a major obstacle that prevents this approach from generalizing to the BMW representations. The forms for the Jones representations found in Theorem \ref{jonesunitary} were proven to be positive definite in a complex neighborhood of 1 by using the fact that the Hecke algebras are a deformation of the complex symmetric algebras. That is at $q=1$, $H_q(n)$ collapses to $\C[\Sigma_n]$. Since $\Sigma_n$ is a finite group, its representations are unitary.  Now in a similar way, the BMW algebras are a deformation of the Brauer algebras. However it is unknown whether the irreducible representations of the Brauer algebras are unitary/sesquilinear or not. So some further investigation into the representation theory of the Brauer algebras could give deeper insight into the conjecture. 

 \begin{ex}\label{ex:bmw2}   For the BMW representation of $B_4$ is given on page 272 of \cite{BW}, the form $J$ is the diagonal matrix with the following diagonal entries.  For notational clarity $L=l$.
\begin{align*}
J_{1,1}=&2\\
J_{2,2}=&  -\frac{2 a \left(L^2+1\right) \left(2 a^2 L-a L^2-a+2 L\right)}{(a-L)^2 (a
   L-1)^2}\\
J_{3,3}=   & \frac{2 \left(L^2+1\right) \left(a^3+L\right) \left(a^3 L+1\right)}{a (a-L)
   (a L-1) \left(2 a^2 L-a L^2-a+2 L\right)}\\   
J_{4,4}=& \frac{2 (a+L) \left(a^5 L^2+a^4 L-a^3 L^2-a^3+a^2 L^3+a^2 L-a
   L^2-L\right)}{a \left(L^2+1\right) \left(a^3+L\right) (a L-1)} \\   
 J_{5,5}=  & \frac{2 (a+L) (a L+1) \left(a^3 L+1\right) \left(2 a^3 L^2+a^3+a^2
   L+a L^2+L^3+2 L\right)}{a \left(L^2+1\right) (a L-1) \left(a^5 L^2+a^4 L-a^3
   L^2-a^3+a^2 L^3+a^2 L-a L^2-L\right)}\\   
J_{6,6}=   & -\frac{2 \left(a^5-L\right) (a+L) (a L+1) \left(a^3
   L+1\right)}{a^3 (a L-1) \left(2 a^3 L^2+a^3+a^2 L+a L^2+L^3+2 L\right)}
\end{align*}

Evaluating $a=i$ and $L=1$ leaves $J=2Id$, giving a point where $J$ is positive definite. Continuity of the determinant implies that $J$ is positive definite in a neighborhood of  $(i,1)$ on the complex torus, though it is difficult to determine explicitly the radius of this neighborhood. Taking the Salem number $S=\frac{1}{2} + \frac{1}{\sqrt{2}} + \frac{1}{2 \sqrt{-1 + 2 \sqrt{2}}}$, specializing $a=S^{15}$ and $L=S^3$ leaves $J$ positive definite at the complex places of $\mathcal{O}_{\mathbb{Q}(S)}$. So this representation is discrete at the specialization  $a=S^{15}$, $L=S^3$.

 \end{ex}



\section{Commensurability }\label{sec:Comm}

The irreducible Jones representations corresponding to rectangular Young diagrams, as in Example \ref{33}, are particularly interesting because they each have only one $(n-1)$-subdiagram. This implies that both the restriction representation and the full representation are irreducible, of the same dimension and use the same form $J$; both representations map into the same unitary group. This situation can be mimicked for the other irreducible Jones representations for the non-rectangular diagrams. The approach is to fix a representation and specialize to two different powers of the same Salem number. The ring of integers $\Ok$ will stay the same, but the defining sesquilinear forms might be very different.

Recall the notation of $K$, $L$, $\mathcal{O}_K$ and $\phi$ from Section \ref{sec:Salem}. In general, fixing a number ring $\Ok$ and dimension $m$, the group $U_m(J,\phi , \Ok)$ is determined by the form $J$.  Notice that $U_m(J,\phi , \Ok)=U_m(\lambda J,\phi , \Ok)$ for every $\lambda\in L$, and that the form $J$ is not completely unique. This motivates that following definition.

 \begin{defn}\label{equivform} Matrices $J$ and $H$  are equivalent over $K$ if $Q^* JQ=\lambda H$ for some $Q\in GL_m(K)$ and $\lambda\in Fix(\phi )$.
\end{defn}

It would be nice if equivalent forms gave rise to \textit{equal} unitary groups, but this is not true in general. However, in the careful scenario that the unitary group is a lattice in $SL(\mathbb{R})$, then changing the form by equivalence yields ``the same" lattice, up to commensurability in the following sense.

\begin{defn}
Two groups $G_1$ and $G_2$ are \textbf{commensurable} if there are finite index subgroups $H_1\subseteq G_1$ and $H_2\subseteq G_2$ so that $H_1$ is conjugate to $H_2$.\end{defn}

\begin{defn}
 A \textbf{lattice} in a semisimple Lie group $G$ is a discrete subgroup of $G$ with finite covolume.
\end{defn}

For our purposes, we will take $G=SL_m(\mathbb{R})$ or $PSL_m(\mathbb{R})$.

\begin{prop}\label{SUcomm}
Assume $SU_m(J_1,\phi ,\Ok)$ and $SU_m(J_2,\phi ,\Ok)$ are lattices in $SL_m(\mathbb{R})$. If $J_1$ and $J_2$ are equivalent over $K$, then $SU_m(J_1,\phi ,\Ok)$ is commensurable to $SU_m(J_2,\phi ,\Ok)$
\end{prop}

\begin{proof}
 Let  $\lambda J_1=Q^*J_2Q$ for some $Q\in GL_m(K)$ and $\lambda\in Fix(\phi )$.  For notational clarity, denote $SU(J_i,\Ok)=SU_m(J_i,\phi ,\Ok)$.
 
 Since scalar multiplication commutes with matrix multiplication, then $M^*JM=J$ if and only if $M^* \lambda J M=\lambda J$. So scaling the form preserves the unitary group, and without loss of generality we may assume $\lambda =1$.
It is easy to see that $M^*JM=J$ if and only if $(Q^*M^*Q^{*-1})(Q^*JQ)(Q^{-1}MQ)=Q^*JQ$, which seems like it implies that $SU(Q^* J_1Q,\Ok)=Q^{-1}SU(J_1,\Ok)Q$. However, since $Q$ has coefficients in $K$,  $Q^{-1}MQ$ may not have coefficients in $\Ok$, so we can only conclude that $Q^{-1}SU(J, \Ok)Q\subseteq SU(Q^*JQ,K)$. To avoid this, we need to pass to a finite index subgroup.

Since $K$ is the ring of fractions of $\Ok$, then there exists $\gamma\in \Ok$ so that $\gamma Q\in M_m(\Ok)$. As a ring of integers of an algebraic extension, $\Ok$ is a Dedekind domain and every quotient is finite. So $\Ok/\langle \gamma^2 \rangle$ is finite and $SU(J_1, \Ok/\langle \gamma^2 \rangle)$ is finite. The kernel $N$ of the quotient map $SU(J_1, \Ok)\rightarrow SU(J_1, \Ok/\langle \gamma^2 \rangle)$ has finite index in $SU(J_1,\Ok)$. 

Any element $B$ in the kernel has the form $B=Id+\gamma^2 A$ for some matrix $A\in M_m(\Ok)$. Inserting $Q^*J_2Q$ for $J_1$ into the equation $B^*J_1B=J_1$,  gives that that $QBQ^{-1}$ fixes the form $J_2$. Because $Q$ has coefficients over $K$,  $QBQ^{-1}$ has coefficients in $K$ and not necessarily in $\Ok$. However, since  $QBQ^{-1}=Id+(\gamma Q)A(\gamma Q^{-1})$, and both $A$ and $\gamma Q$ are integral, then $QBQ^{-1}$ is also integral. Thus $QBQ^{-1}\in SU(J_2,\Ok)$.

Since $SU(J_1,\Ok)$ is a lattice, and $N$ is a finite index subgroup, then $N$ is also a lattice in $SL(\mathbb{R})$ with finite covolume. Thus $QNQ^{-1}$ has finite covolume in $SL(\mathbb{R})$ and is therefore a lattice. So $QNQ^{-1}$ is a sublattice of $SU(J_2,\Ok)$ and must have finite index by Marguls' theorem for lattices.

This shows that $N$ is a finite index subgroup $SU(J_1,\Ok)$ and $Q N Q^{-1}$ is finite index in $SU(J_2,\Ok)$.  \end{proof}

So how does this lattice information apply to the Jones representations? Firstly, after a rescaling and reparameterization, the Jones representations can be made to have determinant $\pm 1$, allowing the image to land in a $PSU(J,\phi , \Ok)$ instead of just $U(J,\phi , \Ok)$. Secondly, an arithmetic group theory result of Harish-Chandra, that is formalized in our setting in Chapter 6 of Witte \cite{DWM}, states that $SU_m(J,\phi , \Ok)$ is a lattice in $SL_m(\mathbb{R})$ under the exact Salem number circumstances as required by Theorem 1.1. So Corollary \ref{Jonesdiscrete} can be restated using this new vocabulary.

\begin{cor}
For each irreducible Jones representation, after a change of parameter,  there are infinitely many Salem numbers $s$ so that specializing $q$ to a powers of $s$ maps into a lattice in $PSL_m(\mathbb{R})$.\end{cor}
\begin{proof}

Let $\rho_q$ be an irreducible Jones representation of dimension $m$. The images of the braid generators under $\rho_q$ have determinant $\pm q^k$  for some $k\in\mathbb{N}$. After a change of variable $q=y^m$ and scaling the generators by $\frac{1}{y^{m-k}}$, this adjusted representation $\tilde{\rho}_y$ maps into $PSU_m(J^y, \mathbb{Z}[y^{\pm 1}])$. 

The subgroup $B^2_n$ of squared braids is a non-central normal subgroup of $B_n$ of finite index. The restriction $\tilde{\rho}_y|$ maps $B^2_n$  into $SU_m(J_y,\mathbb{Z}[y^{\pm 1}])$, and by Theorem \ref{thmdiscrep}, there exists infinitely many Salem numbers $s$ so that the specialization $\rho_s|$ at $y=s$ is discrete. Further by the the result in \cite{DWM}, these specializations make $SU_m(J_s, \Ok)$ lattices in $SL_m(\mathbb{R})$. Finite index arguments imply  $PSU_m(J_s, \Ok)$ is a lattice in $PSL_m(\mathbb{R})$.

\end{proof}

Since the goal is to obtain commensurable lattices as images of our Jones representations, and it is more natural to think of lattices in $SL_m(\mathbb{R})$ instead of in $PSL_m(\mathbb{R})$,  one may simply pass to the finite index subgroup $B_n^2$ and continue to think only about lattices in $SL_m(\mathbb{R})$. To apply Proposition \ref{SUcomm} requires equivalent defining forms. In general, it is difficult to determine when two forms are equivalent. The following theorem  gives a complete classification of the sesquilinear forms in a very specific algebraic setting that applies to the Salem number field scenario.

\begin{thm}[Scharlau\cite{SCHAR}, Ch.10]\label{schar}  If $L$ is a global field and $K=L(\sqrt{\delta})$, sesquilinear forms over $K/L$ are classified by dimension, determinant class and the signatures for those orderings of $L$ for which $\delta$ is negative. \end{thm}

This classification relies on the \textit{determinant class} which is defined here. Recall for a Salem number $s$ the following tower of fields.

\begin{center}
\begin{tikzpicture}
\matrix(a)[matrix of math nodes,
row sep=2em, column sep=3em,
text height=1.5ex, text depth=0.25ex]
{\Q(s)=K&K^\times&\\
 \Q(s+\frac{1}{s})=L&L^\times&(L^\times)^2\subseteq Norm(K^\times)\\
\Q\\};
\path(a-1-1)edge node[right]{2} (a-2-1);
\path(a-2-1) edge node[right]{} (a-3-1);
\path[-stealth](a-1-2) edge node[right]{$Norm$}(a-2-2);
\end{tikzpicture}
\end{center}

The Galois group of $K/L$ is generated by $\phi $ which maps $s\mapsto\frac{1}{s}$. There is a multiplicative group homomorphism $Norm:K^\times\rightarrow L^\times$ given by $Norm(\alpha)=\alpha\alpha^\phi $, where $K^\times=K-\{0\}$.  Notice for $\beta\in L$, $Norm(\beta)=\beta\beta^\phi =\beta^2$. So $(L^\times)^2\subseteq Norm(K)$.

\begin{defn} The \textbf{determinant class} of a sesquilinear form $H$ over $K/L$ is the coset of $\det(H)$ in $K^{\times}/Norm(K^\times)$,
$$[\det(H)]= \det(H)Norm(K) .$$
\end{defn}


Taking $\delta=(s-\frac{1}{s})^2$, $K$ can be rewritten as $K=L(\sqrt{\delta})$. Thus we can restate Scharlau's classification in the specific context of Salem numbers. 

\begin{thm}[Scharlau restated]\label{restate}
 Sesquilinear forms over $K/L$ are classified by dimension, determinant class and the signatures for those orderings of $L$ for which $(s-\frac{1}{s})^2$ is negative. \end{thm}

In odd dimensions, it is very simple to show that \textit{all} sesquilinear forms have the same determinant class, up to scaling. However, for even dimensions, the situation is very unclear.

\begin{prop}\label{samedet}
For every odd dimensional invertible sesquilinear matrices $H$ and $J$ over $K$,   $[\det(H)]=[\det(\lambda J)]$ for $\lambda\in L$.
\end{prop}

\begin{proof}
Let $H$ and $J$ be sesquilinear matrices over $K$ of dimension $2k+1$. Hermitian guarantees both $H$ and $J$ are diagonalizable with diagonal entrees fixed by $\phi $. So, the determinant of both $H$ and $J$ are elements in $L$. Let $d_H$ and $d_J$ denote the nonzero determinants of both $H$ and $J$. Thus 
$$d_H=\frac{d_H}{d_J}d_J\stackrel{mod (L^\times)^2}{\equiv}(\frac{d_H}{d_J})^{2k+1} d_J= \det(\frac{d_H}{d_J}J).$$ Since $(L^\times)^2\subseteq Norm(K)$, then $H$ and $\lambda J$ have the same determinant class for $\lambda=\frac{d_H}{d_J}\in L$.
\end{proof}

As a result, to determine whether two forms of the same odd dimension are equivalent, it suffices only to check that they have the same signatures.




\begin{thm}
For $J_t$ a sesquilinear form that is positive definite for $t$ in a neighborhood $\eta$ of 1, there are infinitely many Salem numbers $s$ and integers $n,m$, so that in all odd dimensions, $SU_{2k+1}(J_{s^n},\phi ,\Ok)$ and $SU_{2k+1}(J_{s^m},\phi ,\Ok)$ are commensurable, discrete groups.
\end{thm}

\begin{proof}

  By Lemma \ref{posreal} there are infinitely many Salem numbers $s$ and integers $n,m$ so that  every complex Galois conjugate of $s^m$ and $s^n$ lie in $\eta$. Fix one such Salem number $s$, and $K,L,$ and $\delta$ as above.
  
  By Theorem \ref{schar}, sesquilinear forms are completely classified by dimension, determinant class, and the signatures for the places of $L$ for which $(s-\frac{1}{s})^2$ is negative. By Proposition \ref{samedet}, $J_{s^n}$ and $\lambda J_{s^m}$ have the same determinant class for $\lambda$ in $L$, namely $\lambda=\frac{\det J_{s^n}}{\det J_{s^m}}$.
 
 Let $\sigma$ be a complex placement of $L$. Then  $\sigma(s^m)$ is a complex Galois conjugates of $s^m$, and similarly for $\sigma(s^n)$ and $s^n$. Since $n$ and $m$ were chosen so that all of the complex Galois conjugate of $s^m$ and $s^n$ have arguments in $\eta$, then  $J_{\sigma(s^m)}$ and $J_{\sigma(s^n)}$ are positive definite. Moreover, $\frac{\det J_{s^n}}{\det J_{s^m}}$  and $\sigma(\frac{\det J_{s^n}}{\det J_{s^m}})$ are both positive, making $\lambda>0$. So regardless of whether $\sigma((s-\frac{1}{s})^2)$ is positive or negative, the forms $J_{\sigma(s^i)}$ have the same signature.
 
 Thus, $J_{s^n}$ is equivalent to $\lambda J_{s^m}$, and $SU(J_{s^n},\phi ,\Ok)$ is commensurable to $SU(J_{s^m},\phi ,\Ok)$. The groups are discrete by Theorem \ref{SUdiscrete}. 
 

\end{proof}

\begin{cor}\label{comgroups}
Let $\rho_t:G\rightarrow SL_{2k+1}(\mathbb{Z}[t,t^{-t}])$ be a group representation with a parameter $t$. Suppose there exists a matrix $J_t$ so that: 
\begin{enumerate}
\item for all $M$ in the image of $\rho_t$, $M^*J_tM=J_t$, where $M^*(t)=M^\intercal (\frac{1}{t})$,
\item $J_t=(J_{\frac{1}{t}})^\intercal $,
\item $J_t$ is positive definite for $t$ in an neighborhood $\eta$ of 1
\end{enumerate}
Then, there exists infinitely many Salem numbers $s$, so that for infinitely many integers $n,m$ the specializations $\rho_{s^m}$ at $t=s^m$ and $\rho_{s^n}$ at $t=s^n$ map into commensurable lattices of $SL_{2k+1}(\mathbb{R})$.
\end{cor}

\begin{example}
The reduced Burau representation of $B_4$ is 3 dimensional and, after the appropriate rescaling to have determinant 1,  satisfies Corollary \ref{comgroups}. So certain powers of the specializations in Example \ref{Burau4} map into commensurable lattices in $SL_3(\mathbb{R})$.
\end{example}

\bibliography{jonesrepisunitarybilbliography}{}
\bibliographystyle{plain}

\end{document}